\documentclass[reqno]{amsart}

\input xy
\xyoption{all}      
\usepackage{epsfig}
\usepackage{color}
\usepackage{amsthm}
\usepackage{amssymb}
\usepackage{amsmath}
\usepackage{amscd}
\usepackage{amsopn}
\usepackage{url}
\usepackage{hyperref}\hypersetup{colorlinks}

\usepackage{mathabx}

\usepackage{eufrak}

\usepackage{color} 

\definecolor{darkred}{rgb}{1,0,0} 
\definecolor{darkgreen}{rgb}{0,0.8,0}
\definecolor{darkblue}{rgb}{0,0,1}

\hypersetup{colorlinks,
linkcolor=darkblue,
filecolor=darkgreen,
urlcolor=darkred,
citecolor=darkgreen}

\numberwithin{equation}{section}

\newcommand{\labell}[1] {\label{#1}}

\numberwithin{equation}{section}
\newtheorem {Theorem}{Theorem}
\numberwithin{Theorem}{section}

\newtheorem {Lemma}[Theorem]    {Lemma}

\newtheorem {Proposition}[Theorem]{Proposition}

\theoremstyle{definition}

\theoremstyle{remark}
\newtheorem{Remark}[Theorem]{Remark}
\newtheorem{Example}[Theorem]{Example}

\expandafter\chardef\csname pre amssym.def at\endcsname=\the\catcode`\@
\catcode`\@=11
\def\undefine#1{\let#1\undefined}
\def\newsymbol#1#2#3#4#5{\let\next@\relax
 \ifnum#2=\@ne\let\next@\msafam@\else
 \ifnum#2=\tw@\let\next@\msbfam@\fi\fi
 \mathchardef#1="#3\next@#4#5}
\def\mathhexbox@#1#2#3{\relax
 \ifmmode\mathpalette{}{\m@th\mathchar"#1#2#3}%
 \else\leavevmode\hbox{$\m@th\mathchar"#1#2#3$}\fi}
\def\hexnumber@#1{\ifcase#1 0\or 1\or 2\or 3\or 4\or 5\or 6\or 7\or 8\or
 9\or A\or B\or C\or D\or E\or F\fi}

\font\teneufm=eufm10
\font\seveneufm=eufm7
\font\fiveeufm=eufm5
\newfam\eufmfam
\textfont\eufmfam=\teneufm
\scriptfont\eufmfam=\seveneufm
\scriptscriptfont\eufmfam=\fiveeufm

\catcode`\@=\csname pre amssym.def at\endcsname

\def    \eps    {\epsilon}

\newcommand{\EE}{{\mathcal E}}
\newcommand{\FF}{{\mathcal F}}

\newcommand{\CA}{{\mathcal A}}

\newcommand{\CS}{{\mathcal S}}

\newcommand{\A}{{\mathcal A}}

\newcommand{\PP}{{\mathcal P}}

\def    \F      {{\mathbb F}}
\def    \HH     {{\mathbb H}}

\def    \C      {{\mathbb C}}
\def    \R      {{\mathbb R}}

\def    \T      {{\mathbb T}}

\def    \12    {{\frac{1}{2}}}

\def    \tr     {\operatorname{tr}}

\def    \SB     {\operatorname{SB}}

\def    \Sp     {\operatorname{Sp}}
\def    \SU     {\operatorname{SU}}
\def    \SO     {\operatorname{SO}}
\def    \Spin     {\operatorname{Spin}}

\def    \H     {\operatorname{H}}

\def    \CL     {\operatorname{CL}}

\def    \ssminus        {\smallsetminus}

\newcommand{\Cl}{{\mathit Cl}}

\begin{document}


\setlength{\smallskipamount}{6pt}
\setlength{\medskipamount}{10pt}
\setlength{\bigskipamount}{16pt}





\title[Arnold Conjecture for Clifford Symplectic Pencils]{Arnold
  Conjecture for Clifford Symplectic Pencils}

\author[Viktor Ginzburg]{Viktor L. Ginzburg}
\author[Doris Hein]{Doris Hein}

\address{Department of Mathematics, UC Santa Cruz,
Santa Cruz, CA 95064, USA}
\email{ginzburg@ucsc.edu}
\email{dhein@ucsc.edu}

\subjclass[2000]{53D40, 32Q15}

\date{\today} \thanks{The work is partially supported by the NSF and
by the faculty research funds of the University of California, Santa
Cruz.}

\begin{abstract} 
  We establish a version of the Arnold conjecture, both the degenerate
  and non-degenerate cases, for target manifolds equipped with Clifford
  pencils of symplectic structures and the domains (time-manifolds)
  equipped with frames of divergence--free vector fields meeting a
  certain additional requirement.  This result generalizes the original
  work on the hyperk\"ahler Arnold conjecture by Hohloch, Noetzel and
  Salamon for three-dimensional time and also the previous work by the
  authors.
\end{abstract}

\maketitle

\tableofcontents

\section{Introduction}
\labell{sec:intro}
We prove an analog of the Arnold conjecture in the setting where the
one-dimensional time, $S^1$ or $\R$, is replaced by a
multi-dimensional time manifold $M$. To be more specific, the time
manifold $M$ is now any closed manifold equipped with a frame of
diversion-free vector fields satisfying a certain regularity
requirement.  The target manifold $W$ (the space) is equipped with a
Clifford pencil (i.e., a linear space with some additional properties)
of symplectic structures.  The space of null-homotopic maps from $M$
to $W$ carries a suitably defined action functional and, in the spirit
of the classical Arnold conjecture, the number of its critical points
is bounded from below by a certain topological invariant of $W$. As in
the Hamiltonian Arnold conjecture (or in Morse and Ljusternik--Schnirelman
theories), this is the sum of Betti numbers in the non-degenerate case
and the cup-length plus one for general Hamiltonians.

Our result generalizes and builds on the results of Hohloch, Noetzel
and Salamon, \cite{HNS}, and the previous results of the authors,
\cite{GH}. In \cite{HNS}, a similar version of the Arnold conjecture
is established in the non-degenerate case for a flat hyperk\"ahler
target manifold and the time manifold $\T^3$ or $\SU(2)$ by means of
Floer theory.  In \cite{GH}, the degenerate case is treated as well
and, furthermore, for $M=\T^r$, the proof of the Arnold conjecture
results is extended to all flat target manifolds carrying a Clifford
pencil of symplectic structures.  Hence, the main new point of the
present paper is a further extension of the argument. We now consider
any closed time manifold equipped with a regular divergence-free
frame, replacing the Lie group structures of $\SU(2)$ and $\T^r$ with
translation-invariant frames. The regularity requirement 
is simply the condition that the critical points of the action
functional for the zero Hamiltonian are exactly the constant maps.

Our proof differs crucially from the Floer theoretic argument in
\cite{HNS} and technically from the argument in \cite{GH}. As in
\cite{GH}, it relies on a finite--dimensional approximation method
combined with Morse or Ljusternik--Schnirelman theory for generating
functions along the lines of \cite{CZ}.  However, in contrast with
previous works using a similar technique with Lie groups serving as time
manifolds (as, e.g., in \cite{CZ,GH}), we do not explicitly determine
the matrix representations for the (unperturbed) $L^2$-gradient
$\partialslash$ of the action functional by means of the Fourier expansion on
$S^1$ or its counterpart for Lie groups provided by the Peter--Weyl
theorem. Instead, we use the ellipticity of $\partialslash$ to
decompose the function space into the sum of finite--dimensional eigenspaces, 
which turns out to be sufficient for the finite--dimensional
reduction method to apply.

\subsection*{Acknowledgments} 
The authors are deeply grateful to Dietmar Salamon for calling their
attention to the regularity problem. They also
would like to thank Jie Qing, Claude Viterbo, and Siye Wu for useful
discussions. A part of this work was carried out while the first
author was visiting the Institute for Advanced Study during the Symplectic
Dynamics program and he would like to thank the Institute for its warm
hospitality and support.

\section{Arnold Conjecture for Symplectic Pencils}
\label{sec:main}

\subsection{Symplectic Pencils} 
A \emph{symplectic pencil} on a
finite--dimensional, real vector space $V$ is a linear subspace
$\CS\subset \bigwedge^2 V^*$, each element of which, except of course
$0$, is a symplectic structure on $V$. Alternatively, we say that
skew-symmetric bilinear forms $\omega_1,\ldots,\omega_r$ on $V$
generate a symplectic pencil $\CS$ when all non-zero linear
combinations $\sum_l\lambda_l\omega_l$, forming $\CS\ssminus 0$, are
symplectic forms. In what follows, we call $\dim \CS$ the \emph{rank}
of the pencil and assume that $\{\omega_1,\ldots,\omega_r\}$ is a
basis of~$\CS$.

Consider, for instance, the Clifford algebra $\Cl_r$ of a
positive-definite quadratic form on $\R^r$ (We refer the reader to
\cite{LM} or, e.g., \cite[Chap.\ 2]{HP}, for a discussion of Clifford
algebras; note that here, in contrast with \cite{GH}, we use the
conventions of \cite{LM}.)  Let $V$ be a real $\Cl_r$-module, i.e., an
(orthogonal) representation of $\Cl_r$. This is simply a collection of
$r$ complex structures $J_1,\ldots,J_r$ on $V$ (corresponding to an
orthonormal basis in $\R^r$), which anti-commute and are all
compatible with the same inner product $\left<\,,\right>$.  In other
words, the operators $J_l$ are $\left<\,,\right>$-orthogonal,
$J_l^2=-I$ for all $l$, and
\begin{equation}
\label{eq:anti-comm}
J_lJ_j+J_jJ_l=0 \text{ whenever $l\neq j$}.
\end{equation}
Then the forms 
$$
\omega_l(X,Y):=\left<J_l X,Y\right>
$$
generate a symplectic pencil $\CS$ of rank $r$ on $V$. To see this,
note that $\omega=\sum\lambda_l\omega_l$ is symplectic if and only if
$J=\sum_l\lambda_l J_l$ is non-degenerate. This is the case when not
all $\lambda_l=0$, since $J^2=-(\sum_l\lambda_l^2) I$ due to
\eqref{eq:anti-comm}. In particular, $\omega\neq 0$, and hence 
$\{\omega_l\}$ is a basis of $\CS$. Thus, the rank of
$\CS$ is indeed $r$. We will refer to these pencils (equipped in
addition with the basis $\omega_l$) as \emph{Clifford} symplectic
pencils.

Among Clifford symplectic pencils are, for example, the symplectic pencils
associated with hyperk\"ahler structures. In this case, $r=3$, and the
complex structures $J_l$ are $\left<\,,\right>$-orthogonal and satisfy
the quaternionic relations, i.e., in addition to \eqref{eq:anti-comm},
we also have $J_1J_2=J_3$.

Clearly, when $V$ admits a symplectic pencil $\CS$ of rank $r$, it
also admits pencils of rank smaller than $r$, e.g., subpencils of
$\CS$. Conversely, in some instances, a given pencil can be extended
to a pencil of higher rank. For example, a Clifford pencil of rank two
extends to a hyperk\"ahler pencil by setting $J_3=J_1J_2$. As a
consequence, whenever $V$ carries a Clifford pencil $\CS$ of rank
three, it also carries a hyperk\"ahler pencil, in general different
from $\CS$, obtained by, e.g., keeping $J_1$ and $J_2$ intact and
replacing $J_3$ with $J_1J_2$. Finally note that, as is not hard to
see, symplectic pencils of rank $r$ form an open subset (possibly
empty) in the Grassmannian of $r$-dimensional linear subspaces in
$\bigwedge^2 V^*$.

A vector space $V$ admits a symplectic pencil of rank $r$ if and only
if the unit sphere in $V$ admits $r$ point-wise linearly independent
vector fields.  To see this, let us view a symplectic pencil $\CS$ as
a pencil of non-degenerate Poisson structures on the dual space
$V^*$. Fix a positive definite quadratic form $K\colon V^*\to R$. Let
$X_l$ be the Hamiltonian vector field of $K$ on the sphere $\Sigma=\{
K=1\}\subset V^*$ with respect to the Poisson structure
$\omega_l$. Then the non-degeneracy of $\sum_l\lambda_l\omega_l$
readily implies that the vector fields $X_l$ are point-wise linearly
independent on $\Sigma$. This shows that the rank of $\CS$ is no
greater than the number of linearly independent vector fields on
$\Sigma$. The opposite inequality comes from Clifford pencils, cf.\
\cite[Chap.\ 12]{Hu}.  Finally, recall how the maximal value $r$ of
linearly independent vector fields on the sphere in $V$ is determined
by the dimension of $V$.  Let $\dim V= 2^{4d+c}b$, where $d\geq 0$ and
$0\leq c\leq 3$ are integers and $b$ is odd. Then $r=8d+2^c-1$; see,
e.g., \cite[Chap.\ 12 and 16]{Hu} and, in particular, pp.\ 152--154
therein.

\begin{Remark} 
  One can construct examples of symplectic pencils which are not
  isomorphic, in the obvious sense, to Clifford pencils. Moreover, a
  dimension count suggests, although we do not have a complete proof
  of this fact, that when $r\geq 3$, the collection of non-Clifford
  pencils contains a set which is open and dense in the space of all
  symplectic pencils. (When $r=2$, the situation appears to be more
  complicated.  It is not inconceivable that Clifford pencils of rank
  two are structurally stable, i.e., every pencil close to such a
  Clifford pencil is also Clifford.)  Pencils of linear symplectic
  structures also arise in the study of manifolds equipped with fat
  fiber bundles introduced in \cite{We:fat} or of fat distributions; see
  \cite[Section 5.6]{Mo} and references therein, and also \cite{FZ}.
  Pencils of complex structures are considered in, e.g., \cite{MS,Jo}.
\end{Remark}

We will revisit the Clifford condition in Section \ref{sec:Clifford}.

\subsection{Action Functional} 
In this section, we will first define the action functional under the
conditions needed for the proof of our main theorem, the Arnold
conjecture, and then discuss the role of these conditions
in Section \ref{sec:discussion}.

\subsubsection{Definition of the Action Functional.} 
\label{sec:action-def}
Throughout this subsection, we will assume that $\omega_1,\ldots,\omega_r$ is a basis
of a Clifford symplectic pencil $\CS$ on $V$ as above. 

Let $W$ be a smooth compact quotient of $V$ by a group of
transformations preserving all $\omega\in \CS$ and the product
$\left<\,,\right>$. (As a consequence, the complex structures $J_l$ are
also preserved.) For instance, $W$ can be the quotient of $V$ by a
lattice. (There are, however, other examples; see, e.g., \cite[p.\
2548]{HNS}.)

Furthermore, let us fix a closed manifold $M$ equipped with a volume
form $\mu$ and a frame $v=\{v_1,\ldots,v_r\}$ of divergence-free
vector fields, i.e., $r$ divergence--free vector fields
$v_i$, $i=1,\ldots,r$, forming a basis of the tangent space at every point
of $M$. The manifold $M$ takes the role of time in Hamiltonian
dynamics.

By analogy with Hamiltonian dynamics, a Hamiltonian is a time-dependent function on the target space, i.e., a smooth function
$$
H\colon M\times W\to \R.
$$
The action functional $\A_H$ is real--valued on the space $\EE$ of
$C^\infty$-smooth (or just $C^2$), null-homotopic maps $f\colon M\to
W$. As in \cite{GH}, we introduce $\A_H$ in two steps. First, let
$F\colon [0,\, 1]\times M\to W$ be a homotopy between $f$ and the
constant map. This is an analog of a capping in the definition of the
standard Hamiltonian action functional. The unperturbed action
functional is
$$
\A(f)=-\sum_l \int_{[0,\, 1]\times M} F^* \omega_l \wedge i_{v_l}\mu.
$$
It is routine to check that $\A(f)$ is well-defined, i.e., independent
of $F$. (Here, it would be sufficient to assume that, e.g., the
universal covering of $W$ is contractible.) Finally, the total or
perturbed action functional is
\begin{equation}
\label{eq:action}
\A_H(f)=\A(f)-\int_M H(f)\mu.
\end{equation}
For instance, when $r=1$ and $M=\T^1$, we obtain the ordinary action
functional of Hamiltonian dynamics. In the case of a hyperk\"ahler
quotient $W$ and $M=\T^3$ or $\SU(2)$, this action functional
coincides, up to a sign, with the ones defined in~\cite{HNS}.

The differential of $\A$ at $f\in\EE$ is 
$$
(d\A)_f (w)=\sum_l\int_M \omega_l(L_{v_l}f,w)\mu,
$$
where $w\in T_f\EE$ is a vector field along $f$. Thus, the gradient of
$\A$ with respect to the natural $L^2$-metric on $\EE$ is a Dirac type operator
$$
\nabla_{L^2} \A (f) =\sum_l J_l L_{v_l} f =: \partialslash f.
$$
Note that, under our assumptions on $v_l$ and $\omega_l$, this
operator is elliptic and self-adjoint; see Proposition
\ref{prop:ellipticity}. This fact is absolutely crucial for the proof
of the main theorem of the paper (Theorem \ref{thm:main}). In what
follows, when the dependence of $\partialslash$ on the frame
$v=\{v_1,\ldots,v_r\}$ is essential, we will use the notation
$\partialslash_v$.

As in \cite{Sa}, let us call the operator $\partialslash$ (or the
frame $v$ when $V$ and the pencil are fixed) \emph{regular} when the
only solutions $f\colon M\to V$ of the equation $\partialslash f=0$
are constant functions. In other words, $W$ is exactly the set of
critical points of $\CA$. This requirement, playing an important role
in the main result of the paper, is further discussed in
Section~\ref{sec:regularity}.

Regardless of the regularity condition, we have
$$
\nabla_{L^2} \A_H (f) = \partialslash f -\nabla H (f),
$$
where $\nabla H$ denotes the gradient of $H$ along $W$.  As a
consequence, the critical points of $\A_H$ are the solutions $f\in\EE$ of
the equation
\begin{equation}
\label{eq:crit-pnts}
\partialslash f =\nabla H (f).
\end{equation}

At a critical point $f$ of $\A_H$, the Hessian $d^2_f \A_H$ is defined
in the standard way as the second variation of $\A_H$. This is a
quadratic form on $T_f\EE$ equal to the $L^2$-pairing with the
linearization of $\nabla_{L^2} \A_H$ at $f$.  We call $f$ a
non-degenerate critical point when this operator $T_f\EE\to T_f\EE$ is
one-to-one, cf.\ \cite[p.\ 2559]{HNS}. A Hamiltonian $H$ is said to be
non-degenerate when all critical points of $\A_H$ are
non-degenerate. In our setting, non-degeneracy is a generic condition
on $H$, i.e., the set of non-degenerate Hamiltonians is residual in
$C^\infty(M\times W)$. (The proof in \cite[p.\ 2574--2576]{HNS}
carries over to our setting with straightforward modifications.)

\subsubsection{Requirements.}
\label{sec:discussion}
Clearly, the action functional $\CA_H$ is defined for any manifold $W$
equipped with $r$ symplectic (or even closed) $2$-forms and a closed
manifold $M$ equipped with a volume form $\mu$ and $r$
divergence--free vector fields. (Of course, as in the case of the
classical Hamiltonian action functional, $\CA_H$ can be multivalued
unless $W$ meets some additional requirements or, in other words, its
value on $f$ may depend on the choice of capping $F$.) However, in
such a general setting, the action functional is probably of little
interest.  From our perspective, one meaningful condition to impose on
$\CA_H$ or, equivalently, $\CA$ is that the operator $\partialslash$
should be elliptic.
To analyze what is essential for this requirement to hold, we focus on
the manifolds $W$ and $M$ separately.

Assume first that the vector fields $v_i$ form a basis at every point
of $M$. Then the operator $\partialslash$ is elliptic if and only if
the forms $\omega_l$ generate a symplectic pencil in the space of
two-forms $\Omega^2(M)$. (However, the forms need not be linearly
independent.) For instance, a manifold with a flat symplectic pencil,
i.e., the quotient $W$ of a vector space $V$ equipped with a
symplectic pencil, meets this requirement.  Note that closed manifolds
admitting (Clifford) symplectic pencils in $\Omega^2(W)$ are extremely
rare, cf. \cite[Chap.\ 21]{GHJ}. When $r\geq 2$, every such
``Clifford'' manifold is automatically hyperk\"ahler with the third
complex structure being $J_1J_2$.  The authors are not aware of any
non-flat example where $r>3$. Yet even flat examples would be of
interest in the context of the Arnold conjecture. In fact, so far all
available methods for $r\geq 2$ require $W$ to be flat; see
\cite{GH,HNS}.

Looking at the other side of the story, let us assume that the forms
$\omega_l$ form a basis of a symplectic pencil. It is easy to see that
$\partialslash$ cannot be elliptic unless the vector fields $v_l$
generate the tangent space to $M$ at every point of $M$, cf.\ the
proof of Proposition~\ref{prop:ellipticity}. However, the operator
$\partialslash$ can still be fairly close to elliptic, e.g.,
hypoelliptic, when the vector fields are bracket--generating.

\subsection{Arnold Conjecture}
Let now $W$ and $M$ be as in Section \ref{sec:action-def}. Namely, $W$
is the quotient of a vector space $V$ equipped with a Clifford pencil
of rank $r$ by a group of transformations preserving the pencil and
the inner product $\left<\,,\right>$; and let $M$ be a closed
$r$-manifold with a volume form $\mu$ and a frame
$v=\{v_1,\ldots,v_r\}$ of divergence-free vector fields.

Denote by $\CL(W)$ the cup-length of $W$, i.e., the maximal number of
elements in $\H_{*>0}(W;\F)$ such that their cup-product is not equal
to zero,  maximized over all fields $\F$.  Likewise, let $\SB(W)$
(the sum of Betti numbers) stand for $\sum_j \dim_\F \H_j(W;\F)$,
maximized again over all $\F$.

In the spirit of the Arnold conjecture, we have

\begin{Theorem}
\label{thm:main}
Let $M$ and $W$ be as above, and furthermore assume that
$\partialslash_v$ is regular.  Then for any Hamiltonian $H$, the
action functional $\A_H$ has at least $\CL(W)+1$ critical points. If
$H$ is non-degenerate, the number of critical points is bounded from
below by $\SB(W)$.
\end{Theorem}

This theorem generalizes several previous results. Namely, when $r=1$ and $M=S^1$, this is the original Arnold conjecture as proved in \cite{CZ}.
When the
target space $W$ is hyperk\"ahler and the domain is either $M=\SU(2)$
or $M=\T^3$, the non-degenerate case of this theorem was originally
proved in \cite{HNS} using a version of Floer theory.  When
$M=\SU(2)$ and $W$ is hyperk\"ahler or when $M$ is an arbitrary torus
and $W$ is as above, both the degenerate and the non-degenerate case
were established in \cite{GH} by means of a finite--dimensional
reduction method originating from \cite{CZ} and similar, up to some
technical details, to the one used here.


\begin{Remark}
  The proof of the theorem does not, in fact, make use of the
  requirement that $\CS$ is a Clifford pencil (beyond some notational
  aspects), and thus the theorem holds for arbitrary symplectic
  pencils. However, in all cases where $\partialslash$ is regular
  known to the authors the pencil is Clifford; see Section
  \ref{sec:regularity}.
\end{Remark}

\begin{Remark}
\label{rmk:non-compact}
As in \cite{GH}, Theorem \ref{thm:main} extends to non-compact
quotients $W$ of $V$ without any significant changes in the
proof. However, in the non-compact case,  certain restrictions must be imposed on the
behavior of the Hamiltonian $H$ at infinity and the lower bounds on
the number of critical points may possibly depend on $H$. For
instance, let us assume that a finite covering $W'$ of $W$ is a
Riemannian product of a flat torus and a Euclidean space $V'$,
e.g., $W$ is an iterated cotangent bundle of a flat
manifold. Then it suffices to require the lift of $H$ to $M\times W'$
to coincide, outside a compact set, with a non-degenerate quadratic form
on $V'$ with constant coefficients. In this case, the lower bounds on
the number of critical points are again $\CL(W)+1$ and, respectively,
$\SB(W)$.
\end{Remark}

\section{Regularity and the Clifford Condition}
\label{sec:regularity+Clifford}

\subsection{Examples of Regular Frames}
\label{sec:regularity}
The analogy between the standard Dirac operator (see \cite{Hi,LM}) and the
operator $\partialslash_v$ suggests that $\partialslash_v$ need not be
regular in general, unless the frame $v$ meets additional
requirements. For $r=3$, the regularity of $\partial_v$ is studied in
\cite{Sa}.  In this section, without attempting to analyze the
regularity question in detail or depth, we briefly discuss some further
examples where regularity or lack thereof can be 
established as a consequence of some other results.

\begin{Example}[Torus]
\label{ex:torus}
Let $M$ be the torus $\T^r$ equipped with the standard volume
form and let the frame $v$ be formed by vector fields with constant
coefficients. (It suffices to assume that the vector fields commute:
$[v_i,v_j]=0$ for all $i$ and $j$.) Then $\partialslash_v$ is regular
for any Clifford pencil.

Indeed, clearly, $\ker\partialslash$ contains the constant functions.
To show that all functions in the kernel are constant, we consider the
operator
\[
\partialslash^2=-\Delta+2\sum_{i<j}J_iJ_j L_{\left[v_i,v_j\right]}=-\Delta,
\]
where $\Delta:=\sum L_{v_l}^2$.  This operator, in contrast with
$\partialslash$, is scalar -- applying $\partialslash^2$ to a function
$f$ amounts to applying it to the components $u$ of $f$ with respect
to some basis in $V$. Furthermore, $\partialslash^2=-\Delta$ is
(second order) elliptic and self-adjoint and the maximum principle
holds for scalar solutions $u\colon M\to\R$ of the equation
$\partialslash^2u=0$, see \cite[Section 6.4]{Ev}. Since the domain $M$
of $u$ is a compact manifold without boundary, $u$ must be
constant. To summarize, once $\partialslash^2 f=0$, every component of
$f$ is a constant function, and hence $f$ is constant. Since, clearly,
$\ker \partialslash^2\supset \ker \partialslash$ (in fact, the two
kernels coincide), we conclude that
$\ker\partialslash$ contains only constant functions.
\end{Example}

\begin{Example}[$\SU(2)$, following \cite{Sa}]
\label{ex:SU(2)}
Let us identify $\SU(2)$ with the unit sphere in the space of
quaternions $\HH$ and let $v$ be the left-invariant frame such that
$v_1=\mathbf{i}$, $v_2=\mathbf{j}$, and $v_3=\mathbf{k}$ at the
identity.  Set $V=\HH$ with 
$J_1=\mathbf{i}$, $J_2=\mathbf{j}$ and $J_3=\mathbf{k}$. Then
$\partialslash_v$ is regular; see \cite{Sa} and also \cite{GH,HNS}.

On the other hand, if we take $v_1=2\mathbf{i}$, $v_2=-\mathbf{j}$
and $v_3=-\mathbf{k}$ at the identity, keeping the rest of the data
the same as above, the operator $\partialslash_v$ fails to be regular. Indeed, as
is observed in \cite{Sa}, the natural inclusion $f\colon \SU(2)\to \HH$
satisfies the equation $\partialslash_v f=0$.
\end{Example}

\begin{Example}[Compact Lie Groups]
\label{ex:Lie-gps}
Let $G$ be a compact $r$-dimensional Lie group with Lie algebra
$\mathfrak{g}$. Fix a bi-invariant metric on $G$ and assume that the
adjoint action lifts to a homomorphism
$G\to\Spin(\mathfrak{g})$. Among the Lie groups with this property are
tori, $\SU(n)$, $\SO(2n)$, the quaternionic unitary groups $\Sp(n)$,
and of course all simply connected compact groups; see, e.g.,
\cite{Sl}. Let $v$ be a left-invariant frame, which is orthonormal
with respect to the bi-invariant metric. (For instance, the frames from
Example \ref{ex:torus} and the first frame considered in Example
\ref{ex:SU(2)} meet this requirement, but the second frame in Example
\ref{ex:SU(2)} does not.)  Then $\partialslash_v$ is regular for any
Clifford symplectic pencil.

To see this, first note that under our conditions on $v$, the operator
$\partialslash_v$ can be identified with the standard Dirac operator
(see, e.g., \cite{LM} for the definition) with respect to the
connection $\nabla$ on $G$ which is compatible with the bi-invariant
metric and for which the frame $v$ is parallel, i.e., $\nabla_\xi
v_i=0$ for all $i=1,\ldots,r$ and all tangent vectors $\xi$; see
\cite{Sl}. (This is not the Levi--Civita connection: $\nabla$ is not
torsion--free unless $G$ is abelian. In fact,
$T(\xi,\eta)=-[\xi,\eta]$, \cite{Sl}, and clearly $\nabla$ is flat.)

By ``additivity'', it suffices to prove the result in the case where
$V$ is an irreducible Clifford module. Furthermore, complexifying $V$
and extending $J_i$'s to $V\otimes\C$ in a complex--linear fashion, we
may assume that $V$ is a complex irreducible representation of
$\Cl_r$.  (When $r$ is even $V$ is unique and when $r$ is odd there are two such
representations; \cite{LM}.)  In this setting the
regularity of the Dirac operator for $\nabla$ is established in
\cite{Sl}.
\end{Example}

\begin{Example}[Generic Regularity]
\label{ex:generic_regularity}
Let us fix $V$ and a Clifford pencil of rank $r$ on $V$. Assume that
$M$ admits a regular frame $v$. Then the set of regular frames
is open and dense in the $C^\infty$-topology in the collection of all
divergence-free frames on $M$.

The
statement that regular frames form an open set is a consequence of the
fact that surjectivity of a Fredholm operator is an open condition.
To show that regular frames are dense, we argue similarly to the proof
of \cite[Lemma 1.3]{Sa}. Let $w$ be an arbitrary frame, which we may
assume to be non-regular. Then we claim that the frame $w+\eps v$ is
regular for small $\eps>0$.  Consider the kernel $K$ of the operator
$\partialslash_w$ on the space of $W^{1,2}$-maps $M\to V$ with zero
mean. Since $v$ is regular, the quadratic form
$\left<f,\partialslash_v f\right>$ is non-degenerate on $K$.  Set
$A(s)=\partialslash_{w+\eps v}$.  The claim follows now from
\cite[Lemma A.2]{Sa} asserting, roughly speaking, that an operator
in a family of self-adjoint operators $A(s)$ is bijective
whenever the quadratic form $\big<f,\dot{A}(0)f\big>$ is
non-degenerate on $\ker A(0)$ and $s>0$ is sufficiently small. 
\end{Example}

\begin{Remark}
  It is interesting to point out that the choice of a pencil plays no
  role in Examples \ref{ex:torus}--\ref{ex:generic_regularity}. In
  fact, we do not have any example of a frame that would be regular
  for some pencils and non-regular for some others.
\end{Remark}

\subsection{Clifford Condition Revisited}
\label{sec:Clifford}

We conclude this discussion with a reformulation of the Clifford
condition, which, we feel, better illuminates it from the
symplectic--geometrical point of view. Namely, we claim that a
symplectic pencil $\CS$ is Clifford if and only if there exists an
inner product $\left<\,,\right>$ on $V$ which is compatible (up to a
factor) with all non-zero forms $\omega\in \CS$. To be more precise,
given an inner product $\left<\,,\right>$ and a two-form $\omega$,
denote by $A_\omega$ the skew-symmetric matrix uniquely determined by
the condition that $\omega(X,Y)=\left<A_\omega X,Y\right>$ for all $X$
and $Y$ in $V$.  We say that $\left<\,,\right>$ is compatible with a
symplectic pencil $\CS$ if for every $\omega\in \CS$ we have
$A_\omega^2=-\lambda I$ for some $\lambda\geq 0$ depending on
$\omega$. In other words, every non-zero form $\omega\in\CS$ becomes
compatible with $\left<\,,\right>$ in the standard sense after, if
necessary, a rescaling. 
Note that when $\omega\neq 0$ we automatically have $\lambda>0$ since
 in this case $\omega$ is non-degenerate.

\begin{Proposition}
\label{prop:metric-Clifford}
A symplectic pencil $\CS$ is Clifford if and only if there exists an
inner product $\left<\,,\right>$ which is compatible with $\CS$.
\end{Proposition}

\begin{proof} In one direction the assertion is obvious: a Clifford
  pencil admits an inner product compatible with it. Proving the
  converse, assume that $\left<\,,\right>$ is compatible with
  $\CS$. For $\omega$ and $\eta$ in $\CS$, set $(\omega,\eta):=-\tr
  (A_\omega A_\eta)$. Clearly, this is a bi-linear symmetric pairing
  on $\CS$, and due to the compatibility condition, $(\,,)$ is
  positive-definite: $(\omega,\omega)=\lambda>0$ when $\omega\neq
  0$. We can rewrite this as
\begin{equation}
\label{eq:norm}
A_\omega^2=-\|\omega\|^2 I,
\end{equation}
where $\|\cdot\|$ stands for the norm on $\CS$ with respect to
$(\,,)$.  In particular, $A_\omega$ is a complex structure if and only if
$\|\omega\|=1$,

To prove that $\CS$ is a Clifford pencil, it suffices to find a basis
$\omega_1,\ldots,\omega_r$ of unit symplectic forms such that the
operators $A_{\omega_i}$ anti-commute. We claim that this is true for
any orthonormal basis. To establish this, it suffices to show that
$A_\omega$ and $A_\eta$ anti-commute whenever $\omega$ and $\eta$ are
orthogonal to each other with respect $(\,,)$. Consider two such
forms, which we can assume to have unit norm, and set
$\sigma=a\omega+b\eta\in \CS$, where $a$ and $b$ are non-zero
scalars. Then, by mutual orthogonality of these forms,
$$
\|\sigma\|^2=a^2\|\omega\|^2+b^2\|\eta\|^2=a^2+b^2,
$$
and thus, by \eqref{eq:norm}, 
$$
A_\sigma^2= -(a^2+b^2)I.
$$
On the other hand, since $A_\sigma=aA_\omega+bA_\eta$, we have
$$
A_\sigma^2= -(a^2+b^2)I+ab(A_\omega A_\eta+A_\eta A_\omega).
$$
It follows that $A_\omega A_\eta+A_\eta
A_\omega=0$, i.e., $A_\omega$ and $A_\eta$ anti-commute.
\end{proof}

\section{Proof of Theorem \ref{thm:main}}
\label{sec:proofs}

The argument follows closely the finite--dimensional reduction method
of Conley and Zehnder, \cite{CZ}, and its version for
multi-dimensional time $M=\SU(2)$ or $\T^r$ introduced in \cite{GH}.
However, we do not work with an explicit expression for
$\partialslash$ obtained via Fourier analysis on a
compact Lie group (which $M$ is not) relying on the Peter--Weyl
theorem. Instead, we expand $f$ in eigenvectors of $\partialslash$ using the
ellipticity of $\partialslash$ to reduce the problem to 
finite--dimensional Morse theory for generating functions. Beyond this
point, the argument is quite standard, and hence we omit here some
straightforward, minor technical details of the proof. 

Let us first assume that $W$ is the quotient of a vector space $V$ by
a lattice and, as a consequence, $W$ is a torus. We will discuss the
modifications needed to deal with the general case at the end of the
proof. (This step is essentially identical to its counterpart in
\cite{GH}.)
 
Since $f$ is null-homotopic, it can be lifted to a map
$\tilde{f}\colon M\to V$, where only the mean value $f_0$ depends on
the lift.  In other words, here we view $\EE$ as an
infinite--dimensional vector bundle over $W$ with projection map
$f\mapsto f_0$. This vector bundle is trivial and its fiber $\FF$ is
canonically isomorphic to the space of smooth maps $M\to V$ with zero
mean. We can regard $\EE$ as a sub-bundle in $W\times L^2_0(M,V)$.

\begin{Proposition}
\label{prop:ellipticity}
The operator $\partialslash$ is elliptic on the space of $V$-valued
functions on $M$ and self-adjoint with respect to the $L^2$-inner
product on this space.
  \end{Proposition}
  
  \begin{proof}
    The operator $\partialslash$ is elliptic if and only if the symbol
    $\sigma(\partialslash)=\sum \lambda_l J_l$ is invertible for all
    non-zero (co)vectors $\lambda=(\lambda_1,\ldots, \lambda_r)$. 
    Essentially by definition, this is true for any metric
    $\left<\,,\right>$, since the forms $\omega_l$ generate a
    symplectic pencil. (In the Clifford case, we
    have $\sigma(\partialslash)^2=-(\sum \lambda_l^2)I$, which also
    establishes ellipticity.)  That $\partialslash$ is self-adjoint can
    be easily shown using Stokes' theorem and the fact that the
    time-manifold $M$ is closed.
  \end{proof}

  The next step in the proof of the theorem is decomposing $\EE$ into
  the eigenspaces of~$\partialslash$. The fact that the eigenvalues go
  to infinity or negative infinity is then sufficient for the
  finite--dimensional reduction to go through.  Note that we can view
  $\partialslash$ either as a linear operator on $\FF$ or as a
  fiberwise linear operator on $\EE=W\times \FF$, independent of the
  point of the base. By the regularity assumption, the
  latter operator is fiberwise non-degenerate, which will 
  be crucial for the proof of the theorem. (This is not just a
  consequence of ellipticity.)

Recall that the spectrum of a self-adjoint elliptic operator is real
and countable and the eigenvalues tend to infinity and/or negative
infinity; see, e.g., \cite[Chapter III]{LM}. Furthermore, the
eigenspaces are finite--dimensional, mutually orthogonal, and the eigenvectors form a complete orthogonal system.

Denote by $\FF_N$ the subspace in $\FF$ spanned by all eigenvectors
of $\partialslash$ for eigenvalues with absolute value not exceeding $N$ and
let $\FF_N^\perp$ be the $L^2$-orthogonal complement of $\FF_N$ in
$\FF$. Thus, $\FF_N^\perp$ is spanned
by all eigenvectors whose eigenvalues have absolute value greater than
$N$.  We can view $\EE_N:=W\times \FF_N$ as a subbundle in $\EE$.  It
will also be useful to regard $\EE$ as a vector bundle over $\EE_N$
with fiber $\FF_N^\perp$.  Denote by $\PP_N$ the (fiberwise)
$L^2$-orthogonal projection of $\EE$ onto $\EE_N$ and by $\PP_N^\perp$
the projection of $\EE=\EE_N\times\FF_N^\perp$ onto the second
component $\FF_N^\perp$.

Our goal is to show that equation \eqref{eq:crit-pnts} has at least
the desired number of solutions. Let $f=g+h$ with $g\in \EE_N$ and
$h\in \FF_N^\perp$. Clearly, $f$ satisfies \eqref{eq:crit-pnts} if and
only if
\begin{equation}
\label{eq:crit-1}
\partialslash g=\PP_N \nabla H(g+h)
\end{equation}
and
\begin{equation}
\label{eq:crit-2}
\partialslash h =\PP_N^\perp \nabla H(g+h).
\end{equation}

Let us focus on the second of these equations with $g$ fixed and both
sides viewed as functions of $h$, cf.\ \cite{CZ} and \cite{GH}. Since
$\partialslash$ is regular and by the definition of $\FF_N^\perp$, the
restriction of the operator $\partialslash$ to $\FF_N^\perp$ is
invertible and has only eigenvalues whose absolute value is greater
than $N$.  Denote the inverse of this restriction by
$\partialslash_N^{-1}$. It is clear that \eqref{eq:crit-2} is
equivalent to the equation
\begin{equation}
\label{eq:crit-3}
h =\partialslash_N^{-1} \PP_N^\perp \nabla H(g+h).
\end{equation}
Note that the right hand side is now defined for all $h$ in the $L^2$-closure
$\bar{\FF}_N^\perp$ of $\FF_N^\perp$, without any smoothness requirement. 
We claim that for sufficiently large $N$ and  for any $g\in \EE_N$, equation
  \eqref{eq:crit-3} (and hence \eqref{eq:crit-2}) has a unique
  solution $h=h(g)$ and this solution is smooth.

To show this, note first that, when $N$ is
sufficiently large, $h\mapsto \partialslash_N^{-1}
\PP_N^\perp \nabla H(g+h)$ is a contraction operator on
$\bar{\FF}_N^\perp$ with respect to the $L^2$-norm. Indeed,
$$
\|\partialslash_N^{-1} \PP_N^\perp \nabla H(g+h_1)
-\partialslash_N^{-1} \PP_N^\perp \nabla
H(g+h_0)\|
\leq O(1/N)\| \nabla H (g+h_1)-\nabla H (g+h_0)\|.
$$
(All norms are $L^2$ unless specified otherwise.) Furthermore, in obvious notation,
\begin{eqnarray*}
\| \nabla H (g+h_1)-\nabla H (g+h_0)\| &=&
\Big\| \int_0^1 \frac{d}{d s} \nabla H (g+ s h_1 +(1-s)
h_0) \, ds \Big\|\\
&\leq& \|\nabla^2 H\|_{L^\infty}\|h_1-h_0\|.
\end{eqnarray*}
Hence,
$$
\|\partialslash_N^{-1} \PP_N^\perp \nabla H(g+h_1)-\partialslash_N^{-1} \PP_N^\perp \nabla
H(g+h_0)\|
\leq O(1/N) \|h_1-h_0\|,
$$
which shows that we can indeed choose $N$ sufficiently large such that the map
$h\mapsto \partialslash_N^{-1} \PP_N^\perp \nabla H(g+h)$ is a
contraction. 

The fact that the fixed point $h=h(g)$ of this operator
is a smooth function is established by the standard
bootstrapping argument and elliptic regularity. Namely,
recall that, since $\partialslash$ is a first order elliptic operator
(see Proposition~\ref{prop:ellipticity}), a solution $h$ of the equation
$\partialslash h= y$ is of Sobolev class $H^{s+1}$ whenever $h$ and
$y$ are $H^s$; see, e.g., \cite[p.\ 193]{LM}. We have 
$\partialslash h=\PP_N^\perp\nabla H (g+h)\in L^2=H^0$, 
and therefore $h\in H^1$. Now, since $H$ and $g$ are smooth,
we also have $\PP_N^\perp\nabla H (g+h)\in H^1$, and hence $h\in H^2$,
etc.

Note also that as an immediate consequence of \eqref{eq:crit-3} and of
the fact that $\nabla H$ is bounded, since $H$ is a function on a
compact manifold, we have
\begin{equation}
\label{eq:h}
\| h(g)\|=O(1/N)\text{ and } \|\partialslash h(g)\| =O(1) \text{
  uniformly in $g$.}
\end{equation}
These estimates will be used in the proof of Lemma
\ref{lemma:quadratic} below.

From a more geometric perspective, $h(g)$ is the unique critical
point of the action functional $\A_H$ on the fiber over $g$ of the
vector bundle $\EE\to\EE_N$. Set $\Phi(g):=\A_H(g+h(g))$. In other
words, $\Phi$ is obtained from $\A_H$ by restricting the action functional
to the section $g\mapsto h(g)$ of this vector bundle, formed by the
fiber-wise critical points. Therefore, $g$ is a critical point of $\Phi$
if and only if $f=g+h(g)$ is a critical point of $\A_H$, i.e., a
solution of \eqref{eq:crit-pnts}, and every critical point of $\A_H$ is
captured in this way. It remains to show that the generating function
$\Phi$ on $\EE_N$ has the required number of critical points.

The key feature of this function is that it is asymptotically (i.e., at
infinity in the fibers of $\EE_N$) a non-degenerate quadratic
form. To be more precise, set
$$
\Phi_0(g)=\A(g)=\left<\partialslash g,g\right>_{L^2} \text{ and } R=\Phi-\Phi_0 .
$$
By definition, $\nabla \Phi_0(g)=\partialslash g$. Hence, as has been
pointed out above, the unperturbed action $\Phi_0$ is a fiberwise
non-degenerate quadratic form since $\partialslash$ is regular. (Here
and throughout the rest of the proof, the metric on $\EE_N=W\times
\FF_N$ is the product of the fiberwise $L^2$-metric and the metric on
$W$.)  Furthermore, the perturbation $R$ is small compared to $\Phi_0$
when $N$ is sufficiently large. To be more precise, we have

\begin{Lemma} Outside a compact set in $\EE_N$, we have
\label{lemma:quadratic}
\begin{equation}
\label{eq:gen-function}
|R|+\|\nabla R\|<\| \nabla \Phi_0\|.
\end{equation}
\end{Lemma}

\begin{Remark}
  Lemma \ref{lemma:quadratic}, combined with the
  non-degeneracy of $\Phi_0$, implies that the
  critical set of $\Phi$ is a compact subset of $\EE_N$. As a
  consequence, the set of solutions of \eqref{eq:crit-pnts} (the
  critical points of $\A_H$) is also compact. Proving this fact
  directly, 
  in the context of the hyperk\"ahler Floer theory, is a rather
  subtle problem; cf.\ \cite{HNS}. This is the reason we have chosen
  to provide a detailed proof of \eqref{eq:gen-function} below.
\end{Remark}

\begin{proof}[Proof of Lemma \ref{lemma:quadratic}]
  To establish \eqref{eq:gen-function}, note first that $H$ and
  $\nabla H$ are bounded, for $H$ is a function on a compact manifold.
  Therefore, the integral of $H$ makes a bounded contribution to $R$
  and $\nabla R$, while the right hand side of \eqref{eq:gen-function}
  grows linearly as $g\to \infty$ in the fiber due to the non-degeneracy
  of $\Phi_0$.  Thus, we can ignore $H$ in \eqref{eq:gen-function} and
  only need to estimate the growth of the difference
\begin{equation}
\label{eq:R_0}
R_0:=\A(g+h(g))-\A(g)=2\left<\partialslash g, h(g)\right>+
\left<\partialslash h(g), h(g)\right>,
\end{equation}
or, to be more precise, of $|R_0|$ together with $\| \nabla R_0\|$. (It
is worth pointing out that this is not equivalent to proving the lemma
in the case where $H=0$. Even when the contributions of $H$ and
$\nabla H$ to $R$ and $\nabla R$ are ignored, $H$ still enters the
problem via the map $h$ which depends on $H$.)

To estimate the growth of $|R_0|$ and $\|\nabla R_0\|$, first observe that 
\begin{equation}
\labell{eq:R0}
| R_0(g) |\leq O(1/N)\big(\|\nabla\Phi_0(g)\|+1\big).
\end{equation}
(Recall our convention from above that all norms, here and below, are $L^2$ unless specified otherwise.) This
is an immediate consequence of \eqref{eq:h}.

In a similar vein, it is not hard to prove that
\begin{equation}
\label{eq:grad-R0}
\| \nabla R_0(g)\|\leq
O(1) + O(1/N)\|\nabla \Phi_0 (g)\|
\end{equation}
by observing that the derivative of the function $g\mapsto h(g)$ is
uniformly bounded by a constant $O(1/N)$, as can be shown by
differentiating \eqref{eq:crit-3} with respect to $g$. (See the
calculation below.) Combining the upper bounds \eqref{eq:R0} and
\eqref{eq:grad-R0} with the fact that $\|\nabla\Phi_0(g)\|$ grows
linearly with $g$ since $\Phi_0$ is non-degenerate, we see that
\eqref{eq:gen-function} holds outside a compact set.

We conclude the argument by giving a detailed proof of
\eqref{eq:grad-R0}. Differentiating \eqref{eq:R_0}, 
we have
\begin{equation}
\label{eq:dR_0}
\begin{split}
dR_0(g)(w) &= 2\left<\partialslash w, h(g)\right>
+2\left<\partialslash g, Dh(g) w\right>\\
& + \left<\partialslash (Dh(g)w), h(g) \right>
+ \left<\partialslash h(g), Dh(g) w \right>,
\end{split}
\end{equation}
where $w\in T_g\EE_N$. To prove \eqref{eq:grad-R0}, we will bound
every term on the right hand side of \eqref{eq:dR_0}.

\emph{The first term:} Using \eqref{eq:h} and the fact that
$\partialslash$ is 
self-adjoint, we have
\begin{equation}
\label{eq:1}
|\left<\partialslash w, h(g)\right>|=|\left<w,\partialslash h(g)\right>|
\leq O(1)\| w\|.
\end{equation}

To deal with the remaining terms, we need to obtain an upper bound on
the norm of the operator $Dh(g)\colon T_g\EE_N\to T_{h(g)}\EE$. By \eqref{eq:crit-3},
 we have, in obvious notation,
\begin{align*}
Dh(g)w &=\frac{d}{ds}\partialslash_N^{-1} \PP_N^\perp \nabla
H\big(g+sw+h(g+sw)\big)
\Big|_{s=0}\\
& = \partialslash_N^{-1} \PP_N^\perp \nabla^2 H\big(g+h(g)\big) w
+ \partialslash_N^{-1} \PP_N^\perp \nabla^2 H\big(g+h(g)\big) Dh(g)w.
\end{align*}
Hence,
$$
\big[I-\partialslash_N^{-1} \PP_N^\perp \nabla^2
H\big(g+h(g)\big)\big] Dh(g)w=
\partialslash_N^{-1} \PP_N^\perp \nabla^2 H\big(g+h(g)\big) w,
$$
and thus
$$
Dh(g)w=\big[I-\partialslash_N^{-1} \PP_N^\perp \nabla^2
H\big(g+h(g)\big)\big]^{-1}
\partialslash_N^{-1} \PP_N^\perp \nabla^2 H\big(g+h(g)\big) w.
$$
Note that as in \eqref{eq:h}, since $H$ is bounded together with all
its derivatives, we have
$$
\left\|\partialslash_N^{-1} \PP_N^\perp \nabla^2
  H\big(g+h(g)\big)\right\|=O(1/N)
$$
uniformly in $g$. Therefore,
\begin{equation}
\label{eq:Dh}
\| Dh(g)w\|=O(1/N) \| w\|.
\end{equation}
Now we are in a position to bound the remaining terms.

\emph{The second term:} Recall that
$\nabla \Phi_0(g)=\partialslash g$. Using \eqref{eq:Dh}, we have
\begin{equation}
\label{eq:2}
|\left< \partialslash g,Dh(g)w\right>|\leq O(1/N)\|\nabla
\Phi_0(g)\|\cdot\|w\|.
\end{equation}

\emph{The third and the fourth terms:} These two terms are equal;
for $\partialslash$ is self-adjoint. Thus, by \eqref{eq:h} and
\eqref{eq:Dh}, we have
\begin{equation}
\label{eq:34}
|\left<\partialslash (Dh(g)w, h(g) \right>|
=|\left<\partialslash h(g), Dh(g) w \right>|\leq
O(1/N)\|w\|.
\end{equation}
Combining the estimates \eqref{eq:1}, \eqref{eq:2} and \eqref{eq:34} for
 the individual terms with
\eqref{eq:dR_0}, we obtain the desired estimate \eqref{eq:grad-R0}.  
This concludes the proof of the lemma.
\end{proof}

An argument similar to the proof of Lemma \ref{lemma:quadratic} shows
that a critical point $g$ of $\Phi$ is non-degenerate if and only if
$f=g+h(g)$ is a non-degenerate critical point of~$\A_H$.

Finally, recall that whenever $\Phi=\Phi_0+R$ is a function on the total
space of a vector bundle over an arbitrary closed manifold $W$ such
that $\Phi_0$ is a fiberwise non-degenerate quadratic form and
\eqref{eq:gen-function} holds, the function $\Phi$ has at least
$\CL(W)+1$ critical points. Moreover, when $\Phi$ is Morse, the number of
critical points is bounded from below by $\SB(W)$. This is a
standard fact and we refer the reader to \cite{CZ} for the original
proof and to, e.g., \cite{We} for a different argument.

The general case, where $W$ is the quotient of $V$ by a group $\Gamma$,
is treated exactly as in \cite{GH}. Namely, first recall that $\Gamma$
contains a finite--index subgroup $\Gamma'$ consisting only of parallel
transports, \cite[p.\ 110]{Wo}.  Thus $W'=V/\Gamma'$ is a torus and
the projection $W'\to W$ is a covering map with a 
finite group of deck transformations, $\Pi=\Gamma/\Gamma'$. The
previous argument applies to the natural lift of the problem to $W'$
and the entire construction is $\Pi$-equivariant. As a result, we
obtain a vector bundle $\EE'_N\to W'$ equipped with a $\Pi$-action
covering the $\Pi$-action on $W'$ and a $\Pi$-invariant function
$\Phi'$ on $\EE'_N$, which is asymptotically quadratic at infinity,
i.e., $\Phi'$ satisfies \eqref{eq:gen-function}. The critical points of
$\A_H$ for the original problem correspond to the $\Pi$-orbits of the
critical points of $\Phi'$.  The quotient $\EE_N=\EE'_N/\Pi$ is a
vector vector bundle over $W$ and the function $\Phi'$ descends to a
function $\Phi$ on $\EE_N$. (The quotient $\EE_N$ and the function
$\Phi$ are smooth, for the $\Pi$-action on $\EE'_N$ is free as an
action covering a free action on $W'$.)  The critical points of $\Phi$
are in one-to-one correspondence with the critical points of $\A_H$
for the original problem and $\Phi$ is also asymptotically quadratic at
infinity. The theorem now follows as before from the lower bounds on
the number of critical points of $\Phi$.


\begin{thebibliography}{BEHWZ}


\bibitem[CZ]{CZ}
C. Conley, E. Zehnder, 
The Birkhoff--Lewis fixed point theorem and a conjecture of V.I. Arnold,
\emph{Invent.\ Math.}, \textbf{73} (1983), 33--49.

\bibitem[Ev]{Ev} L.C. Evans,
\emph{Partial Differential Equations}, Graduate Studies in Mathematics, vol.\ 19,
Second edition, AMS 2010.

\bibitem[FZ]{FZ} L. Florit, W. Ziller, Topological obstructions to
  fatness, Preprint 2010, arXiv:1001.0967.
 
 \bibitem[GH]{GH} V. Ginzburg, D. Hein, 
 Hyperk\"ahler Arnold Conjecture and its Generalizations, 
Preprint 2011, arXiv:1105.0874; to appear in \emph{Internat.\ J. Math.}
 
\bibitem[GHJ]{GHJ} M. Gross, D. Huybrechts, D. Joyce, \emph{Calabi--Yau
    Manifolds and Related Geometries}, Lectures at a Summer School in
  Nordfjordeid, Norway, June 2001, Springer-Verlag, Berlin, Heidelberg,
  2003.
 
\bibitem[Hi]{Hi} N. Hitchin, Harmonic spinors, \emph{Advances in
    Math.}, \textbf{14} (1974), 1--55.

\bibitem[HNS]{HNS} S. Hohloch, G. Noetzel, D. Salamon, Hypercontact
  structures and Floer theory, \emph{Geom.\ Topol.}, \textbf{13}
  (2009), 2543--2617.

\bibitem[HP]{HP}
J.-S. Huang, P. Pand\v zi\'c, \emph{Dirac Operators in Representation 
Theory}, Birkh\"auser Boston, Inc., Boston, MA, 2006.

\bibitem[Hu]{Hu}
D. Husemoller, \emph{Fibre Bundles}, Springer-Verlag, New York, 1994.

\bibitem[Jo]{Jo}
D. Joyce, Manifolds with many complex structures, \emph{Quart.\
  J. Math.\ Oxford Ser.\ (2)}, \textbf{46} (1995), 169--184.

\bibitem[LM]{LM} H.B. Lawson, M.-L. Michelsohn, \emph{Spin Geometry},
  Princeton Mathematical Series, 38, Princeton University Press,
  Princeton, NJ, 1989.

\bibitem[Mo]{Mo} R. Montgomery, \emph{A Tour of Subriemannian
    Geometries, their Geodesics and Applications}, Mathematical
  Surveys and Monographs, vol.\ 91, American Mathematical Society,
  Providence, RI, 2002.

\bibitem[Sa]{Sa} D. Salamon, The three dimensional Fueter equation and
  divergence free frames, Preprint 2012, arXiv:1202.4165.

\bibitem[Sl]{Sl}
S. Slebarski, Dirac operators on a compact Lie group, \emph{Bull.\ 
London Math.\ Soc.}, \textbf{17} (1985), 579--583.

\bibitem[MS]{MS}
A. Moroianu, U. Semmelmann,
Clifford structures on Riemannian manifolds, Preprint 2009,
arXiv:0912.4207.

\bibitem[We1]{We:fat} A. Weinstein, Fat bundles and symplectic
  manifolds, \emph{Adv.\ in Math.}, \textbf{37} (1980), 239--250.

\bibitem[We2]{We}
A. Weinstein,  
$C^{0}$ perturbation theorems for symplectic fixed
points and Lagrangian intersections, in \emph{South Rhone seminar on
geometry, III (Lyon, 1983)}, pp.\ 140--144, \emph{Travaux en Cours}, 
Hermann, Paris, 1984.

\bibitem[Wo]{Wo}
J. Wolf,
\emph{Spaces of Constant Curvature}, Fifth edition,
Publish or Perish, Inc., Houston, TX, 1984.

\end{thebibliography}
\end{document}